\documentclass[12pt,leqno]{amsart}
\usepackage[T1]{fontenc}
\usepackage{amsmath,amssymb,amsthm,amsfonts,mathtools,mathrsfs}
\usepackage[shortlabels]{enumitem}
\usepackage{microtype}
\usepackage{xcolor}
\usepackage[hypertexnames=false]{hyperref}
\usepackage[nameinlink,noabbrev]{cleveref}
\usepackage{mathdots}
\usepackage{xstring}
\usepackage{fontawesome5}

\definecolor{linkblue}{RGB}{35,80,130}
\hypersetup{
  colorlinks=true,
  allcolors=linkblue
}

\newcommand{\gh}[1]{%
  \begingroup
  \StrBehind{#1}{github.com/}[\githubpath]%
  \StrBehind{\githubpath}{/}[\githubname]%
  \href{#1}{\faGithub\,\texttt{\githubname}}%
  \endgroup
}

\allowdisplaybreaks
\numberwithin{equation}{section}

\theoremstyle{plain}
\newtheorem{theorem}{Theorem}[section]
\newtheorem{proposition}[theorem]{Proposition}
\newtheorem{corollary}[theorem]{Corollary}
\newtheorem{lemma}[theorem]{Lemma}

\theoremstyle{definition}
\newtheorem{definition}[theorem]{Definition}
\newtheorem{remark}[theorem]{Remark}
\newtheorem{question}[theorem]{Question}

\renewcommand{\le}{\leqslant}
\renewcommand{\ge}{\geqslant}

\newcommand{\abs}[1]{\lvert#1\rvert}
\newcommand{\norm}[1]{\lVert#1\rVert}
\newcommand{\N}{\mathbb{N}}

\newcommand{\R}{\mathbb{R}}

\newcommand{\one}{\mathbf{1}}
\newcommand{\supp}{\operatorname{supp}}

\newcommand{\parent}{\pi}
\newcommand{\BX}{B_X}
\newcommand{\solid}{\operatorname{solid}}
\newcommand{\solidgen}{\operatorname{solidgen}}
\newcommand{\leanstatement}[3]{%
  \href{https://github.com/pedrotradacete/OrderClosures/blob/f82095757a85fe355d4d11bde3531d5a42f46dd6/#1\#L#2}%
  {\ensuremath{\forall}\,\nolinkurl{#3}}}
\newcommand{\leanformalization}[1]{%
  {\normalfont\footnotesize [#1]}}

\title{Order closure, order adherence and Fatou norms}

\author{A.~Avil\'es}
\address{Universidad de Murcia, Departamento de Matem\'aticas, Campus de Espinardo 30100 Murcia, Spain.}
\email{avileslo@um.es}

\author{M.~A.\ Taylor}
\address{Department of Mathematics\\
ETH Z\"urich, Ramistrasse 101, 8092 Z\"urich, Switzerland.
} \email{mitchell.taylor@math.ethz.ch}

\author{P.~Tradacete}
\address{Instituto de Ciencias Matem\'aticas (CSIC-UAM-UC3M-UCM)\\
Consejo Superior de Investigaciones Cient\'ificas\\
C/ Nicol\'as Cabrera, 13--15, Campus de Cantoblanco UAM\\
28049 Madrid, Spain.}
\email{pedro.tradacete@icmat.es}
\date{}



\begin{document}

\begin{abstract}
We record two results in the theory of vector and Banach lattices related to order adherence. First,
we give a negative answer to Gao and Leung's question on whether the
$uo$-adherence of sublattices must be order closed. Second, we present a
self-contained counterexample showing that a Banach lattice with a weakly Fatou norm
need not admit any equivalent lattice norm with the Fatou property.
\end{abstract}

\maketitle

\section{Introduction}\label{sec:introduction}

Recall that a net $(x_\alpha)$ in a vector lattice $X$ \emph{order converges} to $x\in X$, denoted $x_\alpha\xrightarrow{o}x$, if there exists a net $y_\beta\downarrow 0$ (decreasing with infimum 0) in $X$ such that for all $\beta$ there exists $\alpha_0$ with 
\[
|x_\alpha-x|\leq y_\beta,
\]
for any $\alpha\geq \alpha_0$. The net $(x_\alpha)$ \emph{$uo$-converges} (unbounded order converges) to $x\in X$, denoted $x_\alpha\xrightarrow{uo}x$, if 
\[
|x_\alpha-x|\wedge u\xrightarrow{o}0
\]
for all $u\in X_+$. In spaces of measurable functions, it is well-known that a sequence $uo$-converges to zero if and only if it converges to zero pointwise almost everywhere (i.e.~up to a set of zero measure). On the other hand, Bilokopytov and Troitsky \cite{MR4366912} have recently shown that in spaces of continuous functions $C(K)$, a sequence  $uo$-converges to zero if and only if it converges to zero pointwise on a comeagre set. For further information on order and $uo$-convergence we refer the reader to \cite{MR3666441, MR3817114, Tay, Tay1}. 

Given a notion of convergence $c\in \{o,uo\}$, a subset $Y$ of a vector lattice $X$ is \emph{$c$-closed}  if $\overline{Y}^c=Y$. The \emph{$c$-closure} of $Y$ in $X$ is the smallest $c$-closed subset of $X$ containing $Y$. We will denote the \emph{$c$-adherence}  of $Y$ in $X$ by $\overline{Y}^c:=\{x\in X : \exists\ (y_\alpha)\subseteq Y \ \text{with} \ y_\alpha\xrightarrow{c}x\}$. Our aim here will be to exploit the fact that unlike with topological closures, $\overline{Y}^c$ need not be $c$-closed and that an arbitrarily large number of iterations of $c$-adherence might be needed to reach the $c$-closure of a set. 

    With the motivation of providing a general approach to a fundamental result in financial economics concerning the spanning power of options written on a financial asset, Gao and Leung investigated the question of identifying the order closure of a given sublattice \cite{MR3731703}. Let $X$ be a Banach lattice. By passing to a summable subsequence, it is easy to see that norm convergent sequences have order convergent subsequences. Hence, if $X$ is order continuous  -- meaning that order convergence implies norm convergence in $X$ -- then $\overline{Y}^o$ will be order closed for every subset $Y\subseteq X$. Remarkably, Gao and Leung were able to establish the converse to this statement, even after restricting their attention to sublattices. More specifically, they proved the following theorem.
\begin{theorem}[\cite{MR3731703}, Theorem 2.7]\label{ND}
Let $X$ be a $\sigma$-order complete Banach lattice. The following statements are equivalent.
\begin{enumerate}
    \item The order adherence of every sublattice of $X$ is order closed.
    \item The order and $uo$-adherence of every sublattice agree.
    \item $X$ is order continuous.
\end{enumerate}
\end{theorem}
In addition to \Cref{ND}, they were able to prove in \cite[Lemma 2.1]{MR3731703} that if $Y$ is a sublattice of a vector lattice $X$, then we have the inclusions 
$$\overline{Y}^o\subseteq \overline{Y}^{uo}\subseteq \overline{\overline{Y}^o}^o.$$
Moreover, if $\overline{Y}^{uo}$ is order closed, then it coincides with the order closure of $Y$, and $\overline{Y}^{uo}= \overline{\overline{Y}^o}^o$. In this case, it  follows that $\overline{\overline{\overline{Y}^o}^o}^o=\overline{\overline{Y}^o}^o$. Together with statements (i) and (ii) of \Cref{ND}, these observations motivate the question of whether the $uo$-adherence of a sublattice $Y$ is always order closed. Indeed, this question was explicitly stated in Problem 2.5 of their paper.

\begin{question}[\cite{MR3731703}, Problem 2.5]
Is $\overline{Y}^{uo}$ order closed for every sublattice of a vector lattice $X$?
\end{question}

In \Cref{sec:gao-leung}, we  give a strong negative resolution to this problem by showing that there is no universal bound on the number of iterated order adherences one must take before reaching an order closed sublattice.  This is done by translating into the Banach lattice framework the analogous construction for Boolean algebras first due to Gaifman \cite{zbMATH03205463} and Hales  \cite{zbMATH03194564} and later simplified by Solovay  \cite{MR186598}. The remainder of this section contains a discussion on the size of the order adherence of a solid set $S$ in terms of the size of a generating set of $S$. In particular, we will show in \Cref{thm:solid-iterations} that solid sets may also need an arbitrarily large number of order adherence iterations to reach their order closure.

The second purpose of the paper is to present a self-contained solution to a problem originally posed by D.~Fremlin \cite{fremlinlist} on weakly Fatou norms and resolved by M.~Elliot in \cite{Elliott}. Recall that a lattice norm $\|\cdot\|$ on a vector lattice $X$ has the \emph{Fatou
property} if
\[
\|x\|=\sup_\alpha \|x_\alpha\|
\]
whenever $0\le x_\alpha \uparrow x$ in $X$.  It is \emph{weakly Fatou} if
there exists $K\ge 1$ such that
\[
\|x\|\le K \sup_\alpha \|x_\alpha\|
\]
whenever $0\le x_\alpha \uparrow x$. Clearly, every equivalent renorming of a Fatou norm is weakly Fatou. Fremlin's question is whether the converse holds:
\begin{question}[\cite{fremlinlist}, Problem AB]
   Is every weakly Fatou norm equivalent to a Fatou norm?
\end{question}
We emphasize that the solution to the above problem is known. Indeed, a solution was announced by A.~Wickstead and credited to M. Elliott \cite{Elliott} at the Positivity IX conference in Edmonton, Alberta, 2017. However, the proof has never been published or publicly released. In Section \ref{sec:fremlin}, we present a self-contained construction of a weakly Fatou norm on a Banach lattice which is not equivalent to a Fatou norm.

The connection between Fremlin's problem and order adherence is very natural: If $\|\cdot\|$ is a weakly Fatou norm on $X$, then there is some $K>0$ such that the unit ball satisfies 
\[
\overline{B_{(X,\|\cdot\|)}}^o\subset K B_{(X,\|\cdot\|)}.
\]
Iterating this process $n$ times, we get that
\[
\overline{\overline{B_{(X,\|\cdot\|)}}^{o^{\iddots}}}^o\subset K^n B_{(X,\|\cdot\|)}.
\]
On the other hand, if $\|\cdot\|$ is Fatou we must have $\overline{B_{(X,\|\cdot\|)}}^o=B_{(X,\|\cdot\|)}$. 

In our approach, we will produce for every natural number $n$ a norm on a space of continuous functions which is weakly Fatou with constant 2, but whose $n$-th iterated order adherence of the unit ball has an element of norm $2^n$. This will force any Fatou norm on the space to have a bad equivalence constant with the original norm. Gluing these spaces together for $n\in \mathbb N$ we obtain the claimed construction. 

Our motivation for giving an example of a solid set whose order adherence is not order closed was to correct a mistake made in \cite{Tay}, which wrongly assumed that an increasing supremum of increasing supremum could be written as an increasing supremum. We thank E.~Bilokopytov and K.~Abela for pointing out this mistake to the author. After several months of work, we produced the material in the first two sections of the present paper. Given the above connection to Fremlin's problem, it was natural to ask whether our ideas could also give a simpler solution to this problem. However, constructing the norm so that the solid set whose order adherence keeps expanding is the unit ball is highly non-trivial. To overcome this, we rely heavily on ideas from \cite{Elliott}. The ideas in \cite{Elliott} are ingenious, but the execution is quite complicated and indirect. Our main contribution is to simplify and distill these ideas to produce a Lean-verified and easily digestible proof that we believe will be useful to the community.

\subsection{Lean formalization}

The results presented in this paper have been formalized in
Lean 4 using version \verb|v0.1.0| of
\gh{https://github.com/davidmunozlahoz/banlat}, a Lean library for
Banach lattices. This version of the library, in turn, depends on Mathlib version
\verb|v4.30.0| \gh{https://github.com/leanprover-community/mathlib4}.
Each of the results in \Cref{sec:gao-leung} and \Cref{sec:fremlin} includes a link to the corresponding
Lean declaration pointing to a specific line in
\gh{https://github.com/pedrotradacete/OrderClosures}. For completeness we have also included formalization of Solovay's construction from \cite{MR186598}, which is fundamental for the proof of \Cref{prop:gao-counterexample}, as well as the formalization of \Cref{ND} from \cite{MR3731703}.

\section{The Gao-Leung problem on $uo$-adherences of sublattices}\label{sec:gao-leung}

The next result gives a strong negative answer to the Gao-Leung problem.

\begin{proposition}\label{prop:gao-counterexample}
\leanformalization{\leanstatement{OrderClosures/GaoLeungProblem/Counterexample.lean}{32}{gao_counterexample}}
For any cardinal $\kappa$ there exists a compact Hausdorff space $K$ such
that $C(K)$ is order complete and has density character at least $\kappa$,
together with a separable norm closed sublattice $Y\subseteq C(K)$ such that
the only order closed sublattice of $C(K)$ containing $Y$ is $C(K)$ itself.
\end{proposition}


\begin{remark}\label{rem:gao-cardinality}
To see that \Cref{prop:gao-counterexample} answers the question of Gao and Leung, we include a simple cardinality estimate.
Let $D$ be a subset of a vector lattice $X$ of cardinality $\sigma$, and let
$\overline{D}^o$ denote the order adherence of $D$ in $X$. We claim that
\[
\abs{\overline{D}^o}\le 2^{2^\sigma}.
\]
Fix $x\in \overline{D}^o$ and let
\[
\mathcal{D}(x)=\{A\subseteq D : x\in \overline{A}^o\}\in\mathscr{P}(D).
\]
Since $\mathcal{D}(x)$ is a family of subsets of $D$, we have
$\abs{\mathcal{D}(x)}\le 2^\sigma$. Now suppose $x,y\in \overline{D}^o$ and
$x\ne y$. We prove that $\mathcal{D}(x)\ne \mathcal{D}(y)$. This yields
\[
\abs{\overline{D}^o}
=\abs{\{\mathcal{D}(x):x\in \overline{D}^o\}}
\le \abs{\mathscr{P}(\mathscr{P}(D))}
=2^{2^\sigma}.
\]

To prove the claim, let $(x_\alpha)$ be a net in $D$ such that
$x_\alpha\xrightarrow{o}x$. Let $y_\alpha\downarrow 0$ be such that
$\abs{x_\alpha-x}\le y_\alpha$ eventually. Although it is not strictly
necessary, we may assume that $(y_\alpha)$ has the same index set as
$(x_\alpha)$ after passing to a subnet; see \cite{MR4533923}. Since
$x\ne y$, there exists $\alpha_1$ such that $\abs{x-y}\not\le y_{\alpha_1}$.
Let
\[
A=\{x_\alpha:\alpha\ge \alpha_1\}.
\]
Then clearly $A\in \mathcal{D}(x)$. We show that $A\notin \mathcal{D}(y)$.
Suppose towards a contradiction that there is a net $(a_\gamma)$ in $A$ with
$a_\gamma\xrightarrow{o}y$. Passing to a subnet, choose $z_\gamma\downarrow
0$ with $\abs{a_\gamma-y}\le z_\gamma$ for all $\gamma$. Then
\[
\abs{x-y}\le \abs{x-a_\gamma}+\abs{y-a_\gamma}\le y_{\alpha_1}+z_\gamma.
\]
Passing to the limit gives $\abs{x-y}\le y_{\alpha_1}$, a contradiction.
Thus $\mathcal{D}(x)\ne \mathcal{D}(y)$, as claimed.
\end{remark}


\begin{proof}[Proof of \Cref{prop:gao-counterexample}]
We begin with the complete Boolean algebra $\mathbb{B}$ of Solovay
\cite{MR186598} with $\abs{\mathbb{B}}\ge \kappa$ that contains a countable
subset $G\subseteq \mathbb{B}$ with the property that there does not exist a
complete Boolean subalgebra $\widetilde{\mathbb{B}}$ satisfying
$G\subseteq \widetilde{\mathbb{B}}\subsetneq \mathbb{B}$.

Let $K$ be the Stone space of $\mathbb{B}$. A Boolean algebra is complete if
and only if its Stone space is extremally disconnected, which is equivalent
to $C(K)$ being order complete. We represent
$G\subseteq \mathbb{B}\subseteq C(K)$ in the standard way by indicator
functions. Since $\abs{\mathbb{B}}\ge \kappa$, the space $C(K)$ has density
character at least $\kappa$, because the elements of $\mathbb{B}$ are all
$1$-separated.

Let $Y$ be the closed sublattice of $C(K)$ generated by $G$ and $\one$.
Then $Y$ is separable. Since $C(K)$ is order complete, every order closed
sublattice of it is order complete and norm closed: norm convergent
sequences have subsequences that converge in order. It therefore suffices to
show that if $X$ is an order closed and norm closed sublattice of $C(K)$
containing $Y$, then $X=C(K)$.

Set
\[
\widetilde{\mathbb{B}}=\mathbb{B}\cap X.
\]
Clearly $G\subseteq \widetilde{\mathbb{B}}$. We claim that
$\widetilde{\mathbb{B}}$ is a complete Boolean subalgebra of $\mathbb{B}$.
Once this is proved, the defining property of $G$ yields
$\widetilde{\mathbb{B}}=\mathbb{B}$. Hence
$\mathbb{B}\subseteq X\subseteq C(K)$. Since $\one\in X$ and indicator
functions of clopen sets separate points of the Stone space,
Stone--Weierstrass gives $X=C(K)$.

It remains to prove that $\widetilde{\mathbb{B}}$ is complete. The Boolean
operations are inherited from $\mathbb{B}$, so only completeness requires an
argument. Let $A\subseteq \widetilde{\mathbb{B}}$, and let $b=\sup A$ in the
complete Boolean algebra $\mathbb{B}$. Viewed inside $C(K)$, the indicator
function of $b$ is the order supremum of the family of indicator functions of
elements of $A$. Since those indicator functions belong to the order closed
sublattice $X$, their supremum also belongs to $X$. Thus $b\in
\widetilde{\mathbb{B}}$, and $\widetilde{\mathbb{B}}$ is complete.
\end{proof}


In everything that follows (including \Cref{sec:fremlin}) we shall apply order adherence only to solid sets. Note that the order adherence of a solid set is also solid. For simplicity, we will introduce some notation -- we leave it as a standard exercise to show that the definition of order adherence below is consistent with the one above.

\begin{definition}
Let $X$ be a vector lattice and let $A\subseteq X$ be solid. The
\emph{order adherence} $A^o$ is the solid hull of the set of all $x\in X_+$
for which there exists an upward directed set $B\subseteq A\cap X_+$ with
$\sup B=x$ in $X$.

Inductively, put
\[
A^{o(0)}:=A,
\qquad
A^{o(m+1)}:=\bigl(A^{o(m)}\bigr)^o.
\]
\end{definition}

Let $\solid(A)$ be the least solid set that contains $A$. Namely, 
\[
\solid(A)=\bigcup_{a\in A}[-|a|,|a|].
\] 
Given a solid set $S$, let $\solidgen(S)$ be the least cardinality of a set $A$ such that
$S=\solid(A)$.

\begin{proposition}\label{prop:solid-large-order-closure}
\leanformalization{\leanstatement{OrderClosures/GaoLeungProblem/Counterexample.lean}{141}{exists_solid_large_orderAdherence}}
For any cardinal $\kappa$ there exists a solid set $S$ with
$\solidgen(S)=\omega$ but $\solidgen({S}^o)\ge \kappa$.
\end{proposition}

\begin{proof}
Consider the following Banach sublattice $X$ of $\ell_\infty(\Gamma\times
\omega)$:
\[
X=\left\{f\in \ell_\infty(\Gamma\times\omega) : \exists x\in c_0(\Gamma)\
\exists y\in c_0(\mathbb{N}) : |f(\gamma,n)|\le x(\gamma)+y(n)\right\}.
\]

For every $n$, the characteristic function $\chi_{\Gamma\times\{n\}}$ belongs
to $X$, because
\[
\chi_{\Gamma\times\{n\}}(\gamma,n)\le 0(\gamma)+\chi_{\{n\}}(n).
\]
Let $S$ be the solid subset of $X$ generated by the countable set
$\{\chi_{\Gamma\times\{n\}} : n\in\mathbb{N}\}$. Since $S$ contains all
finitely supported sequences, every element of $X$ belongs to the order
closure, and even to the order adherence, of $S$.

We now estimate $\solidgen(X)$. Given $f\in X$, notice that
\[
\Gamma_f=\{\gamma\in \Gamma : \chi_{\{\gamma\}\times\mathbb{N}}\le f\}
\]
is a finite subset of $\Gamma$. This is because
$\chi_{\Delta\times\mathbb{N}}\notin X$ if $\Delta$ is infinite. Therefore,
if $X=\solid(A)$, then we must cover $\Gamma$ with $\abs{A}$ many finite
sets. This implies that $\abs{A}\ge \abs{\Gamma}$.
\end{proof}

There is, however, a cardinality bound on the size of the order adherence.

\begin{proposition}\label{prop:cardinalitybound}
\leanformalization{\leanstatement{OrderClosures/GaoLeungProblem/Counterexample.lean}{403}{solid_orderAdherence_cardinality_bound}}
If $S$ is a solid set, then
\[
\abs{{S}^o}\le 2^{|S|}.
\]
\end{proposition}

\begin{proof}
It is enough to check that
\[
{S}^o \subseteq
\left\{ x : x^+,x^-\in\left\{\bigvee A : A\subseteq S^+\right\}\right\}.
\]
For this, it is enough to prove that the right-hand side is order closed.
Since taking positive and negative parts are order continuous operations, it
suffices to prove that
\[
Z= \{\bigvee A : A\subseteq S^+\}
\]
is order closed. Suppose $z = o\text{-}\lim z_i$ with $z_i\in Z$. In
particular, we have an increasing net $y_k \leq z_i$ with $\sup y_k = z$.
By taking positive parts, we may suppose that $0\le y_k$.

It remains to show that each $y_k\in Z$. But $0\le y_k\le z_i\in Z$ for some
$i$. So if $z_i = \sup A$, then
\[
y_k = \sup\{a\wedge y_k : a\in A\}.
\]
To justify this, suppose that $w<y_k$ were an upper bound of
$\{a\wedge y_k : a\in A\}$. Then $z_i +w-y_k < z_i$ would be an upper bound of
$A$ below the supremum. Indeed, if $a\in A$, then
\[
a = a - a\wedge y_k + a\wedge y_k \le a - a\wedge y_k + w
\]
and
\[
a-a\wedge y_k = |a\wedge z - a\wedge y_k| \le |z-y_k| = z-y_k.
\]
\end{proof}

We now analyze the number of iterations of order adherence that might be
needed to attain the order closure.

\begin{theorem}\label{thm:solid-iterations}
\leanformalization{\leanstatement{OrderClosures/GaoLeungProblem/Iterations.lean}{612}{solid_sets_require_arbitrarily_many_iterations}}
For every cardinal $\kappa$ with $\aleph_0\leq\kappa$ and every ordinal
$\xi\leq \kappa^+$ there exists a solid set $S$ generated by a set of size
$\kappa$ such that $\xi$ many iterations of order adherence are needed to
reach its order closure.
\end{theorem}

\begin{proof}
As a matter of notation, throughout this proof superscripts will be used to
denote Cartesian products of sets. For example, $I^\omega$ is the set of all
tuples $(i_n)_{n<\omega}$ with $i_n\in I$ for all $n$. But when the base is
$\omega$, it will mean ordinal power, not Cartesian product. For example,
$\omega^2 = \sup\{\omega,\omega+\omega,\ldots\}$ is a countable ordinal, not a
set of tuples.

We define a partial order relation on ordinals. Suppose that we are given
ordinals $\zeta \leq \zeta'$. Let us write them in Cantor normal form
\[
\zeta = \sum_{i=1}^n\omega^{\beta_i}c_i
\qquad \text{and} \qquad
\zeta' = \sum_{i=1}^m\omega^{\gamma_i}d_i.
\]
We say that $\zeta\prec \zeta'$ if $m\leq n$, $\beta_i = \gamma_i$,
$c_i=d_i$ for $i<m$, and either $\gamma_m = \beta_m$ and $c_m< d_m$, or
$\beta_m < \gamma_m < \beta_{m-1}$; if $m=1$ there is no $\beta_{m-1}$, and
the last condition must be interpreted as $\beta_m<\gamma_m$. Some basic
properties are:
\begin{itemize}
\item[P1.] $\preceq$ is a partial order contained in the usual ordinal order
$\leq$.
\item[P2.] For every $\zeta$, the relations $\prec$ and $<$ coincide on
$\{\zeta' : \zeta\prec \zeta'\}$.
\end{itemize}

Now fix an ordinal $\xi$. Consider the compact space
\[
K = \{(x_\zeta)_{\zeta\leq \omega^\xi + 1} : \zeta \prec \zeta' \Rightarrow
x_\zeta \leq x_{\zeta'}\} \subset \{0,1\}^{\omega^\xi + 1}.
\]
Let $\pi_\zeta\in C(K)$ be the projection onto the $\zeta$-th coordinate, and
let $S$ be the solid set of $C(K)$ generated by the $\pi_\zeta$ for
$\zeta$ a successor ordinal. We claim that $\xi$ many iterations of the
order adherence are needed to reach the order closure of $S$. This proves the
statement for $\xi<\kappa^+$. We will deal with the special case
$\xi=\kappa^+$ at the end of the proof.

By the definition of $K$, we have $\pi_\zeta\leq \pi_{\zeta'}$ whenever
$\zeta\prec \zeta'$. Moreover:

\medskip
\noindent
Claim 1: If $Z$ is an infinite pairwise $\prec$-incomparable family of
ordinals, then $\inf\{\pi_\zeta: \zeta\in Z\} =0$.

\smallskip
\noindent
Proof of claim: Suppose $0 \leq f$ is a lower bound. In particular this means
that $f(x)=0$ whenever $x_\zeta = 0$ for some $\zeta\in Z$. Thus, it is
enough to prove that
\[
D = \{x\in K : \exists \zeta\in Z \ x_\zeta = 0 \}
\]
is dense in $K$. So take $y\in K$ and a basic neighborhood of the form
\[
W = \{x\in K : x|_F= y|_F\}
\]
for finite $F\subset \omega^\xi + 1$. We must find $x\in W\cap D$. There
must exist $\zeta_0\in Z$ such that $\zeta'\not\preceq \zeta_0$ for all
$\zeta'\in F$ because elements $\prec$-above $\zeta'$ are linearly ordered by
$\prec$. Define $x$ by declaring $x_\zeta = 0$ if $\zeta\preceq \zeta_0$ and
$x_\zeta = y_\zeta$ otherwise.

\medskip
\noindent
Claim 2: If $Z$ is a nonempty $\prec$-totally ordered family of ordinals with
$\sup(Z) = \alpha$, then
\[
\sup\{\pi_\zeta : \zeta\in Z\} = \pi_\alpha.
\]

\smallskip
\noindent
Proof of claim: It is clear that $\pi_\alpha$ is an upper bound. The
nontrivial case is when the supremum is not a maximum. Let $f\in C(K)$ be
another upper bound. If $\pi_\alpha\not\leq f$, then there exists $x\in K$
such that $f(x) < x_\alpha$. Since $f\geq 0$, we must have $x_\alpha=1$. By
continuity in the product topology, there exists a finite set
$F \subseteq\omega^\xi$ such that $y|_F = x|_F$ implies $f(y) < 1$. Take
$\zeta_0\in Z$ such that $\max(F\cap \alpha) < \zeta_0$. Consider $y\in K$
such that $y_\zeta= 1$ if $\zeta_0\preceq \zeta$ and $y_\zeta=x_\zeta$
otherwise. Then
\[
1 = y_{\zeta_0} \leq f(y) < 1,
\]
a contradiction.

Let $Z_\beta$ be the set of all ordinals $\zeta\leq \omega^\xi$ whose
minimal exponent in Cantor normal form is strictly less than $\beta$ in the
usual ordinal order. Thus $Z_1$ is the set of all successor ordinals. Let
\[
S_\beta = \left\{f\in C(K) : |f|\leq \pi_\zeta \text{ for some }
\zeta\in Z_\beta \right\}.
\]

\medskip
\noindent
Claim 3: $o\text{-}adh^\beta(S) = S_{1+\beta}$.

\smallskip
\noindent
Proof of claim:
We prove it by induction on $\beta$. For $\beta=0$, this is the definition of
$S$. The limit case is trivial. For the successor case, it is enough to prove
that
\[
o\text{-}adh(S_\gamma) = S_{\gamma+1}
\]
for every $\gamma\leq \omega^\xi+1$. The inclusion $[\supseteq]$ follows
from Claim 2 because every $\zeta\in Z_{\gamma+1}$ can be found as a supremum
of a $\prec$-increasing set from $Z_\gamma$. For the inclusion
$[\subseteq]$, suppose that we have an upward directed net
$\{g_i\}\subseteq S_\gamma$ such that $|f| = \sup_i g_i$. For every $i$ we
know that
\[
A_i := \left\{\zeta\in Z_\gamma : g_i\leq \pi_\zeta\right\}\neq \emptyset.
\]
Let $M_i$ be the set of $\prec$-minimal elements of $A_i$. Since it extends
the ordinal order, $\prec$ is a well-founded relation, which implies that
every element of $A_i$ is $\prec$-above an element of $M_i$. Moreover, by
Claim 1, each $M_i$ is finite. If $i\leq j$ then $A_j\subseteq A_i$, and
hence, since $\prec$ is a linear order above any given element,
$|M_j|\leq |M_i|$. Therefore, by restricting from an index onward, we may
suppose that all sets $M_i$ have the same cardinality $k$ and write
$M_i = \{\mu_i^1,\dots,\mu_i^k\}$ in such a way that
$\mu_i^t \preceq \mu_j^t$ when $i\leq j$. But then, since
$g_i\leq \pi_{\mu_i^1}$, we have
\[
|f|\leq \sup_i \pi_{\mu_i^1} = \pi_{\sup_i\mu^1_i}
\]
and $\sup_i\mu^1_i\in Z_{\gamma+1}$.

After Claim 3, it is enough to notice that
$\pi_{\omega^\gamma}\in S_{\gamma+1}\setminus S_\gamma$. It is trivial that
$\pi_{\omega^\gamma}\in S_{\gamma+1}$. If we had
$\pi_{\omega^\gamma}\in S_{\gamma}$, then $\pi_{\omega^\gamma}\leq \pi_\zeta$
for some $\zeta\in Z_\gamma$. We would have $\omega^\gamma\not\preceq \zeta$.
If we consider $x\in K$ with $x_\alpha=0$ if $\alpha\preceq \zeta$ and
$x_\alpha=1$ otherwise, then
\[
\pi_{\omega^\gamma}(x) = 1 < 0 = \pi_\zeta(x),
\]
a contradiction. This finishes the case $\xi<\kappa^+$.

For the case of $\kappa^+$, consider, for every $\xi<\kappa^+$, the solid
set $S_\xi$ generated by a set $\{f^\alpha_\xi : \alpha<\kappa\}$ of elements
of norm one in the Banach lattice $X_\xi$ constructed above, where $\xi$
many iterations of order adherence were needed to reach the order closure.
Take $X$ to be the $\ell_\infty$-sum of all those $X_\xi$, and let $S$ be
the solid set generated by the elements
$f_\alpha = (f_\alpha^\xi)_{\xi<\kappa^+}$.
\end{proof}

The ordinal $\kappa^+$ in \Cref{thm:solid-iterations} is likely not optimal. Although we will not pursue this issue in detail, we include a lemma that might be useful to tackle this question.

\begin{lemma}\label{lem:solid-generated-order-adh}
\leanformalization{\leanstatement{OrderClosures/GaoLeungProblem/Iterations.lean}{643}{solid_generated_orderAdherence}}
If $S = \solid(G)$ for some set $G$ of positive elements, then
\[
S^{o}
=\left\{x : \exists \text{ net } (a_i)\subset G \text{ with }
|x| = o\text{-}\lim_{i\in I} |x|\wedge a_i
= \sup_{i\in I} |x|\wedge a_i\right\}.
\]
\end{lemma}

\begin{proof}
A first observation is that, for a given $x$, the fact that
\[
|x| = o\text{-}\lim_{i\in I} |x|\wedge a_i
= \sup_{i\in I} |x|\wedge a_i
\]
is equivalent to the simultaneous satisfaction of
\[
x^+ = o\text{-}\lim_{i\in I} x^+\wedge a_i
= \sup_{i\in I} x^+\wedge a_i,
\]
\[
x^- = o\text{-}\lim_{i\in I} x^-\wedge a_i
= \sup_{i\in I} x^-\wedge a_i.
\]
For $(\Leftarrow)$ use that $|x| = x^+ + x^-$, and for $(\Rightarrow)$ use
the property checked at the end of \Cref{prop:cardinalitybound}.

Now we proceed to prove the equality of sets. The inclusion $[\supseteq]$
follows from the previous observation. For $[\subseteq]$, suppose
$x=o\text{-}\lim_{i\in I} x_i$ with $x_i\in S$. Then
$|x| = o\text{-}\lim |x_i|$. In particular, there is an increasing net
$\{y_j\}$, which we may suppose consists of positive elements, which is
eventually below $|x_i|$ and has supremum $|x|$. For every $j$, we may thus choose
$a_j$ such that $y_j \leq a_j\in G$.
\end{proof}

\section{A simpler solution to Fremlin's problem on weakly Fatou norms}\label{sec:fremlin}


Recall that a lattice norm $\norm{\cdot}$ on a vector lattice $X$ has the
\emph{Fatou property} if
\[
\norm{x}=\sup_\alpha \norm{x_\alpha}
\]
whenever $0\le x_\alpha\uparrow x$ in $X$, and it is \emph{weakly Fatou} if
there exists $K\ge 1$ such that
\[
\norm{x}\le K \sup_\alpha \norm{x_\alpha}
\]
whenever $0\le x_\alpha\uparrow x$ in $X$.

The goal of this section is to prove the following theorem.

\begin{theorem}\label{thm:fremlin-main}
\leanformalization{\leanstatement{OrderClosures/WeaklyFatou/FinalSpace.lean}{572}{exists_weaklyFatou_not_equivalent_fatou}}
There exists a Banach lattice with a weakly Fatou norm which is not equivalent to any
lattice norm with the Fatou property.
\end{theorem}

The proof has two parts. For each $n\in\N$ we construct a separable Banach
lattice $X_n$ such that
\begin{enumerate}[label=\textup{(\roman*)}]
\item $X_n$ is weakly Fatou with constant $2$;
\item the $n$-th iterated order adherence of the unit ball of $X_n$ contains an
element of norm $2^n$.
\end{enumerate}

We then take the $c_0$-sum of the spaces $X_n$. If that sum admitted an
equivalent Fatou norm, then every iterated order adherence of its unit ball
would stay uniformly bounded, contradicting \textup{(ii)}.

We will try to make the construction as explicit as possible. We will consider the finite-height
tree
\[
G_n:=\bigcup_{k=0}^n \N^k,
\]
and characteristic functions of simple cylinder sets in
the product space $\N^{H_n}$, where $H_n=G_n\backslash \N^n$ is the set of non-terminal nodes in $G_n$. 
This provides a family of functions in which every parent will be the supremum of its children, while a certain norm will still assign the value $2^{-k}$ to every node on level $k$. Establishing the uniform weak
Fatou estimate for this norm will be the most technical part of the argument.



Let us begin with some simple properties of order adherence.

\begin{lemma}\label{lem:order-basic}
\leanformalization{\leanstatement{OrderClosures/WeaklyFatou/Reductions.lean}{25}{iteratedOrderAdherence_mono_and_scale}}
Let $A,B\subseteq X$ be solid and let $\lambda>0$.
\begin{enumerate}[label=\textup{(\alph*)}]
\item If $A\subseteq B$, then $A^{o(m)}\subseteq B^{o(m)}$ for all $m\ge 0$.
\item One has $(\lambda A)^o=\lambda A^o$, hence
$(\lambda A)^{o(m)}=\lambda A^{o(m)}$ for all $m\ge 0$.
\end{enumerate}
\end{lemma}

\begin{proof}
These claims are essentially immediate from the definition.
\end{proof}

The following lemma translates the Fatou and weak Fatou properties into statements about how much the unit ball can expand after iterated order adherence.

\begin{lemma}\label{lem:fatou-order}
\leanformalization{
  \leanstatement{OrderClosures/WeaklyFatou/Reductions.lean}{73}{weakFatou_iterated_unitBall}
  and
  \leanstatement{OrderClosures/WeaklyFatou/Reductions.lean}{133}{fatou_iterated_unitBall}}
Let $(X,\norm{\cdot})$ be a normed vector lattice.
\begin{enumerate}[label=\textup{(\alph*)}]
\item If $\norm{\cdot}$ is weakly Fatou with constant $K$, then
\[
\BX^{\,o(m)}\subseteq K^m \BX
\qquad (m\ge 1).
\]
\item If $\norm{\cdot}$ has the Fatou property, then
\[
\BX^{\,o(m)}=\BX
\qquad (m\ge 1).
\]
\end{enumerate}
\end{lemma}

\begin{proof}
It is enough to treat one order adherence step.

For \textup{(a)}, let $x\in \BX^o$. Choose an upward directed
$B\subseteq \BX\cap X_+$ with $|x|\le \sup B$. By the weak Fatou property, we have
\[
\norm{x}\le \norm{\sup B}\le K \sup_{b\in B}\norm{b}\le K.
\]
Hence, $\BX^o\subseteq K\BX$ and iteration gives the claim.

For \textup{(b)}, the inclusion $\BX\subseteq \BX^o$ is trivial.
Conversely, if $x\in \BX^o$, choose $B$ as above. The Fatou property gives
\[
\norm{\sup B}=\sup_{b\in B}\norm{b}\le 1,
\]
so again $\norm{x}\le 1$. Thus, $\BX^o=\BX$ and the iterated identity
follows.
\end{proof}

In our construction, it will be enough to work with a sequential version of Fatou/weakly Fatou norms. For convenience, we recall the following definitions.

\begin{definition}
Let $(X,\norm{\cdot})$ be a normed vector lattice and let $K\ge 1$.
\begin{enumerate}[label=\textup{(\alph*)}]
\item We say that $K$ is a \emph{weak sequential Nakano constant} for $X$ if
for every increasing sequence $(x_m)$ in $X_+$ with an upper bound in $X$
and with $\sup_m \norm{x_m}\le 1$, and for every $\varepsilon>0$, there
exists an upper bound $y\in X_+$ of the sequence such that
$\norm{y}\le K+\varepsilon$.
\item We say that $K$ is a \emph{weak Nakano constant} for $X$ if the same
statement holds with increasing sequences replaced by upward directed subsets
of the unit ball of $X_+$ having an upper bound in $X$.
\end{enumerate}
\end{definition}

We shall use the following elementary reduction essentially based in \cite[Lemma 2.5]{AzouziRjebT}, which we include for completeness.

\begin{proposition}\label{prop:separable-reduction}
\leanformalization{\leanstatement{OrderClosures/WeaklyFatou/Reductions.lean}{317}{separable_weakSequentialNakano_implies_weakNakano}}
Let $Y$ be a separable normed vector lattice and $K\ge1$ a weak sequential Nakano constant for $Y$. Then $K$ is a weak Nakano constant for $Y$. In particular, $Y$ is weakly Fatou with constant $K$.
\end{proposition}

\begin{proof}
Let $A\subseteq Y_+$ be upward directed, order bounded, and satisfy
$\sup_{a\in A}\norm{a}\le 1$. Since $Y$ is separable, $A$ has a countable
dense subset $D$. Let $C$ be the set of all finite suprema of elements of
$D$. Then $C$ is countable, upward directed, and has the same upper bounds
as $A$: every upper bound of $A$ clearly bounds $C$, while if $y$ bounds $D$
then $y-d\in Y_+$ for every $d\in D$, hence also $y-a\in Y_+$ for every
$a\in A$ by norm closedness of the positive cone. Enumerating $C$ as
$(c_m)$ and replacing it by the increasing sequence
$d_m:=c_1\vee\cdots\vee c_m$, we obtain an increasing sequence with the same
set of upper bounds as $A$. By hypothesis, for every $\varepsilon>0$ there
is $y\in S$ such that $d_m\le y$ for all $m$ and $\norm{y}\le K+\varepsilon$.
Hence $y$ is an upper bound of $A$. This is exactly $K$ being a weak Nakano constant for $Y$.
\end{proof}

Let us now introduce the notation for the underlying tree and corresponding space of functions on it. We fix $n\in\N$ and put
\[
G_n:=\bigcup_{k=0}^n \N^k,
\qquad
H_n:=\bigcup_{k=0}^{n-1} \N^k.
\]
We will refer to $H_n$ as the set of non-terminal nodes of the tree $G_n$. The unique element of $\N^0$ is denoted by $\varnothing$ and will be called
the root. If $t\in \N^k$ and $m\in\N$, write $t^\frown m\in \N^{k+1}$ for
concatenation. We define the parent map $\parent:G_n\to G_n$ by
\[
\parent(\varnothing):=\varnothing,
\qquad
\parent(t^\frown m):=t.
\]
For $t=(t_1,\dots,t_k)$ and $0\le j\le k$, write
$t|j:=(t_1,\dots,t_j)$, with $t|0=\varnothing$.
Let
\[
I_n:=\N^{H_n},
\]
equipped with the product topology, each copy of $\N$ carrying the discrete
topology. An element $\alpha\in I_n$ can be considered as a function which assigns to each non-terminal node
$u\in H_n$ a natural number $\alpha(u)$.

For $t\in G_n$ of length $k$, define the \textsl{cylinder}
\[
E_t:=\{\alpha\in I_n : \alpha(t|j-1)\le t_j \text{ for } 1\le j\le k\},
\]
and set $E_{\varnothing}:=I_n$. Let $C_b(I_n)$ denote the space of bounded continuous functions on $I_n$ and for $t\in G_n$ let us consider the corresponding characteristic function
\[
s_t:=\chi_{E_t}.
\]

\begin{lemma}\label{lem:basic-tree}
\leanformalization{
  \leanstatement{OrderClosures/WeaklyFatou/FiniteTree.lean}{130}{treeCylinder_isClopen}
  and
  \leanstatement{OrderClosures/WeaklyFatou/FiniteTree.lean}{171}{treeFunction_child_properties}}
For every $t\in G_n$, the following properties hold.
\begin{enumerate}[label=\textup{(\alph*)}]
\item $E_t$ is clopen in $I_n$, hence $s_t\in C_b(I_n)$;
\item if $|t|<n$ and $m\in\N$, then $E_{t^\frown m}\subseteq E_t$, hence
$s_{t^\frown m}\le s_t$;
\item if $|t|<n$, then
\[
s_t=\sup_{m\in\N} s_{t^\frown m}
\]
pointwise on $I_n$.
\end{enumerate}
\end{lemma}

\begin{proof}
Part \textup{(a)} is immediate because $E_t$ is defined by finitely many
coordinate inequalities.

Part \textup{(b)} is immediate from the definition.

For \textup{(c)}, note that
\[
E_{t^\frown m}=E_t\cap \{\alpha\in I_n : \alpha(t)\le m\}.
\]
As $m$ increases these sets are increasing, and their union is $E_t$.
Taking characteristic functions gives the claim.
\end{proof}

The following combinatorial fact will play a key role in the construction.

\begin{lemma}\label{lem:finite-cover}
\leanformalization{\leanstatement{OrderClosures/WeaklyFatou/FiniteTree.lean}{224}{treeCylinder_finite_cover}}
Let $t\in G_n$ and let $F\subseteq G_n$ be finite. If
\[
E_t\subseteq \bigcup_{u\in F} E_u,
\]
then there exists $u\in F$ with $E_t\subseteq E_u$.
\end{lemma}

\begin{proof}
Assume that $E_t\not\subseteq E_u$ for every $u\in F$. We shall construct
$\alpha\in E_t\setminus \bigcup_{u\in F} E_u$.

For every proper ancestor $t|j$ of $t$ with $0\le j<|t|$, set
\[
\alpha(t|j):=t_{j+1}.
\]
Now let $a\in H_n$ be any non-terminal node which is not a proper ancestor
of $t$ but is a parent of some element of $F$. Since $F$ is finite, the set
\[
M_a:=\{m\in\N : a^\frown m \text{ is an initial segment of some } u\in F\}
\]
is finite. Choose $\alpha(a)>\max M_a$. On all remaining coordinates set
$\alpha(a):=1$.
By construction, $\alpha\in E_t$.
Take $u=(u_1,\dots,u_k)\in F$. Since $E_t\not\subseteq E_u$, one of the
following two cases occurs.

If the parent $u|k-1$ is a proper ancestor of $t$, then necessarily
$k\le |t|$ and $u_i=t_i$ for $1\le i\le k-1$. If $u_k\ge t_k$, then every
point of $E_t$ satisfies the defining inequalities for $E_u$, contrary to
the assumption that $E_t\not\subseteq E_u$. Hence $u_k<t_k$, and therefore
\[
\alpha(u|k-1)=t_k>u_k.
\]
So, $\alpha\notin E_u$.

If the parent $u|k-1$ is not a proper ancestor of $t$, then by construction
$u|k-1$ belongs to the exceptional set of coordinates on which $\alpha$ was
chosen to dominate all child labels appearing in $F$. In particular,
\[
\alpha(u|k-1)>u_k,
\]
so again $\alpha\notin E_u$.

Thus $\alpha\notin E_u$ for every $u\in F$, a contradiction.
\end{proof}

The next lemma will later allow us to ignore pieces supported on
bands that do not recur. To make this precise, we need a definition.

\begin{definition}
Two subsets $A,B\subseteq G_n$ are \emph{$\parent$-disjoint} if
$\parent[A]\cap \parent[B]=\varnothing$.
\end{definition}

\begin{lemma}\label{lem:pi-disjoint}
\leanformalization{\leanstatement{OrderClosures/WeaklyFatou/FiniteTree.lean}{402}{parentDisjoint_treeFunctions_iInf}}
Let $(F_m)$ be a sequence of finite subsets of $G_n$ which are pairwise
$\parent$-disjoint. Then
\[
\inf_{m\in\N}\ \sup_{t\in F_m} s_t = 0
\]
in the vector lattice $C_b(I_n)$.
\end{lemma}

\begin{proof}
Put $g_m:=\sup_{t\in F_m} s_t=\chi_{U_m}$, where
$U_m:=\bigcup_{t\in F_m} E_t$. It suffices to show that
$\bigcap_m U_m$ has empty interior.

Suppose, towards a contradiction, that some non-empty basic open cylinder
$\Omega\subseteq I_n$ is contained in every $U_m$. Then $\Omega$ fixes only
finitely many coordinates, say those in a finite set $Q\subseteq H_n$.

Because the sets $\parent[F_m]$ are pairwise disjoint and finite, we can
choose $m$ such that $\parent[F_m]\cap Q=\varnothing$ and
$\varnothing\notin F_m$.

Choose any $\alpha\in \Omega$. For every $u\in \parent[F_m]$, change the
value of $\alpha(u)$ to a number strictly larger than every child label
appearing above $u$ in $F_m$. Since no such $u$ lies in $Q$, the modified
point still belongs to $\Omega$; denote it by $\beta$.

Now take $t\in F_m$ and let $u:=\parent(t)$. By construction,
$\beta(u)>t_{|t|}$, so $\beta\notin E_t$. Hence $\beta\notin U_m$, which
contradicts $\Omega\subseteq U_m$.
\end{proof}

We are now in position to define a lattice norm with the required properties.
Let $W_n:=c_{00}(G_n)$, the lattice of finitely supported real functions on
$G_n$. For $t\in G_n$, write $e_t$ for the characteristic function of the
singleton $\{t\}$.

For $w\in W_n$, we define the weight
\[
\rho_n(w):=\sum_{t\in G_n} 2^{-|t|}\,\abs{w(t)}.
\]
Also, let us consider the linear map $T_n:W_n\rightarrow C_b(I_n)$ given by
\[
T_n w:=\sum_{t\in G_n} w(t)s_t
\]
for $w\in W_n$.
For $x\in C_b(I_n)$, set
\[
p_n(x):=\inf\{\rho_n(w):w\in W_n^+,\, |x|\le T_n w\}.
\]

\begin{lemma}\label{lem:pn-seminorm}
\leanformalization{\leanstatement{OrderClosures/WeaklyFatou/TreeNorm.lean}{344}{treeLatticeSeminorm}}
$p_n$ is a lattice seminorm on $C_b(I_n)$.
\end{lemma}

\begin{proof}
Monotonicity and the lattice property are immediate from the definition.
Positive homogeneity is obvious.

For subadditivity, take $x,y\in X_n$ and $\varepsilon>0$. Choose
$u,v\in W_n^+$ with $|x|\le T_n u$, $|y|\le T_n v$,
$\rho_n(u)<p_n(x)+\varepsilon$, and
$\rho_n(v)<p_n(y)+\varepsilon$.
Then
\[
|x+y|\le |x|+|y|\le T_n(u+v),
\]
so
\[
p_n(x+y)\le \rho_n(u+v)=\rho_n(u)+\rho_n(v)
<p_n(x)+p_n(y)+2\varepsilon.
\]
Letting $\varepsilon\downarrow 0$ completes the proof.
\end{proof}

\begin{lemma}\label{lem:norm-comparison}
\leanformalization{\leanstatement{OrderClosures/WeaklyFatou/TreeNorm.lean}{390}{treeSeminorm_norm_comparison}}
For every $x\in C_b(I_n)$ one has
\[
p_n(x)\le \|x\|_\infty \le 2^n p_n(x).
\]
\end{lemma}

\begin{proof}
Because $|x|\le \|x\|_\infty s_{\varnothing}$ and
$\rho_n(e_{\varnothing})=1$, the first inequality is immediate.

For the second, let $w\in W_n^+$ satisfy $|x|\le T_n w$. Since each $s_t$ is
$\{0,1\}$-valued, we have
\[
\|x\|_\infty\le \|T_n w\|_\infty
\le \sum_{t\in G_n} w(t)
\le 2^n \sum_{t\in G_n} 2^{-|t|} w(t)
=2^n \rho_n(w).
\]
Taking the infimum over all such $w$ proves the claim.
\end{proof}

The next lemma gives the exact values of $p_n$ on the tree functions.

\begin{lemma}\label{lem:exact-basis}
\leanformalization{\leanstatement{OrderClosures/WeaklyFatou/TreeNorm.lean}{455}{treeSeminorm_exact_basis}}
Let $t\in G_n$. If $w\in W_n^+$ satisfies $s_t\le T_n w$, then
$\rho_n(w)\ge 2^{-|t|}$. Consequently,
\[
p_n(s_t)=2^{-|t|}.
\]
\end{lemma}

\begin{proof}
The upper bound $p_n(s_t)\le 2^{-|t|}$ is given by $w=e_t$.

For the reverse inequality, let $w\in W_n^+$ satisfy $s_t\le T_n w$, and let
\[
U:=\{u\in \supp w : E_t\subseteq E_u,\ u\ne t\}.
\]
Move all coefficients supported on $U$ down to $t$:
\[
w':=
w+\Bigl(\sum_{u\in U} w(u)\Bigr)e_t-\sum_{u\in U} w(u)e_u.
\]
If $\alpha\in E_t$, then $E_t\subseteq E_u$ for every $u\in U$, so
$T_n w'(\alpha)=T_n w(\alpha)\ge 1$. Thus $s_t\le T_n w'$. Moreover,
every $u\in U$ satisfies $|u|\le |t|$. Indeed, if $|u|>|t|$, then the last
defining inequality for $E_u$ involves a coordinate which is not constrained
by membership in $E_t$, so $E_t\subseteq E_u$ is impossible. Hence
$2^{-|t|}\le 2^{-|u|}$ and $\rho_n(w')\le \rho_n(w)$.

No strict upper bound of $t$ remains in $\supp w'$. Therefore,
\Cref{lem:finite-cover}, applied to the finite family
$\supp w'\setminus\{t\}$, shows that
\[
E_t\not\subseteq \bigcup_{u\in \supp w',\,u\ne t} E_u.
\]
Choose $\alpha\in E_t$ outside that union. Evaluating at $\alpha$ gives
\[
1=s_t(\alpha)\le T_n w'(\alpha)=w'(t),
\]
because every term other than $t$ vanishes at $\alpha$. Hence
\[
\rho_n(w)\ge \rho_n(w')\ge 2^{-|t|} w'(t)\ge 2^{-|t|}.
\]
This proves the claim.
\end{proof}

Let $X_n$ be the closed sublattice of $C_b(I_n)$ generated by the
countable family $\{s_t:t\in G_n\}$, equipped with the restriction of $p_n$.

\begin{corollary}\label{cor:Yn-basic}
\leanformalization{\leanstatement{OrderClosures/WeaklyFatou/TreeNorm.lean}{625}{component_basic}}
$(X_n,p_n)$ is a separable Banach lattice. Moreover,
$p_n(\one)=1$ and every terminal node $t\in \N^n$ satisfies
$p_n(s_t)=2^{-n}$.
\end{corollary}

\begin{proof}
Separability is clear from the countable generating family, and completeness
follows from \Cref{lem:norm-comparison} because $X_n$ is closed in the
supremum norm. The formulae for $p_n(\one)$ and the terminal nodes are
immediate from \Cref{lem:exact-basis}.
\end{proof}


We now focus on the technical part of estimating the weak Fatou constant of $(X_n,p_n)$.

For each non-terminal node $u\in H_n$, let
\[
C_u:=\{u^\frown m : m\in\N\},
\]
and let $B_u\subseteq W_n$ be the band of functions supported on $C_u$.
Also put $B_{\varnothing}:=\R e_{\varnothing}$. Then
\[
W_n=B_{\varnothing}\oplus \bigoplus_{u\in H_n} B_u
\]
as an algebraic $\ell^1$-sum with respect to $\rho_n$.

For each such band $B$, let $P_B:W_n\to B$ denote the coordinate
projection. If $\Lambda$ is a finite family of these bands, write
\[
P_\Lambda:=\sum_{B\in\Lambda} P_B.
\]

Define the upshift $S_n:W_n\to W_n$ by
\[
S_n e_{\varnothing}:=e_{\varnothing},
\qquad
S_n e_{t^\frown m}:=e_t.
\]

\begin{lemma}\label{lem:upshift-basic}
\leanformalization{\leanstatement{OrderClosures/WeaklyFatou/Bands.lean}{279}{treeUpshift_basic}}
For every $w\in W_n^+$ one has
\[
\rho_n(S_n w)\le 2\rho_n(w)
\qquad\text{and}\qquad
T_n w\le T_n S_n w.
\]
\end{lemma}

\begin{proof}
The first inequality follows from
$2^{-|\parent(t)|}=2\cdot 2^{-|t|}$ for $t\ne\varnothing$.

For the second, it is enough to check it on the basis vectors
$e_{t^\frown m}$. There it reads $s_{t^\frown m}\le s_t$, which is
\Cref{lem:basic-tree}\textup{(b)}. The claim follows by linearity and
positivity.
\end{proof}

Let $(x_m)$ be an increasing sequence in $X_n^+$ with $p_n(x_m)\le 1$ for
all $m$. Fix $\varepsilon>0$.

For each $m$ choose $w_m\in W_n^+$ such that
\[
x_m\le T_n w_m
\qquad\text{and}\qquad
\rho_n(w_m)<1+\varepsilon/4.
\]
Since $(x_m)$ is increasing, every subsequence of $(T_n w_m)$ still dominates the
whole sequence $(x_m)$.

\begin{lemma}\label{lem:sharp-subsequence}
\leanformalization{\leanstatement{OrderClosures/WeaklyFatou/Bands.lean}{369}{tree_sharp_subsequence}}
After passing to a subsequence, we may suppose that for every band
$B\in\{B_{\varnothing}\}\cup\{B_u:u\in H_n\}$ the sequence
$\bigl(\rho_n(P_B w_m)\bigr)$ converges.
\end{lemma}

\begin{proof}
There are only countably many bands, and each projected norm is bounded by
$\rho_n(w_m)$. A standard diagonal argument gives the claim.
\end{proof}

Fix such a subsequence and write
\[
\lambda_B:=\lim_{m\to\infty} \rho_n(P_B w_m)
\]
for each band $B$. Since $\rho_n$ is the $\ell^1$-sum of the band norms, we have
\[
\sum_B \lambda_B \le \liminf_{m\to\infty}\rho_n(w_m)\le 1+\varepsilon/4.
\]

\begin{lemma}\label{lem:trim}
\leanformalization{\leanstatement{OrderClosures/WeaklyFatou/Bands.lean}{420}{tree_trim}}
There exists a new sequence $(\widetilde w_m)$ in $W_n^+$ such that:
\begin{enumerate}[label=\textup{(\roman*)}]
\item $x_m\le T_n \widetilde w_m$ for all $m$;
\item $\rho_n(\widetilde w_m)<1+\varepsilon/2$ for all $m$;
\item only finitely many bands in the decomposition given above meet $\supp \widetilde w_m$ for infinitely
many values of $m$.
\end{enumerate}
\end{lemma}

\begin{proof}
Put $M:=2^n$. Choose a finite family $\Lambda$ of bands, containing
$B_{\varnothing}$, such that
\[
\sum_{B\notin \Lambda} \lambda_B < \frac{\varepsilon}{8M}.
\]
Enumerate the remaining bands as $(D_j)_{j\in\N}$ and set
\[
\mu_j:=\lambda_{D_j}+\frac{\varepsilon}{16M\,2^j}.
\]
Then
\[
\sum_{j=1}^\infty \mu_j < \frac{\varepsilon}{4M}.
\]
For each $j$ choose $N_j\in\N$ such that
\[
\rho_n(P_{D_j} w_m)<\mu_j
\qquad (m\ge N_j).
\]
After increasing the $N_j$ if necessary, we may assume that $(N_j)$ is
strictly increasing.

Now define
\[
\widetilde w_m
:=
w_m-\sum_{j:\,N_j\le m} P_{D_j}w_m
+M\Bigl(\sum_{j:\,N_j\le m}\mu_j\Bigr)e_{\varnothing}.
\]
This is a positive vector.
Let $z\in B_{D_j}^+$ with $\rho_n(z)\le \mu_j$. Since every coefficient of
$z$ has weight at least $2^{-n}=M^{-1}$, we have
\[
\sum_{t\in G_n}|z(t)|\leq 2^n\sum_{t\in G_n}2^{-|t|}|z(t)|\leq M\mu_j.
\]
Therefore
\[
T_n z \le M\mu_j\,\one = T_n(M\mu_j e_{\varnothing}).
\]
Applying this with $z=P_{D_j}w_m$ shows that
$T_n\widetilde w_m\ge T_n w_m\ge x_m$.
Also,
\[
\rho_n(\widetilde w_m)
\le \rho_n(w_m)+M\sum_{j=1}^\infty \mu_j
<1+\frac{\varepsilon}{4}+\frac{\varepsilon}{4}
=1+\frac{\varepsilon}{2}.
\]
Finally, if $D_j\notin \Lambda$ and $m\ge N_j$, then
$P_{D_j}\widetilde w_m=0$. Hence, only bands in $\Lambda$ can occur
infinitely often.
\end{proof}

Replace $(w_m)$ by the sequence given by \Cref{lem:trim}. Let $\Lambda$ be
the finite family of bands that still occur infinitely often.

\begin{lemma}\label{lem:thinning}
\leanformalization{\leanstatement{OrderClosures/WeaklyFatou/Moderated.lean}{15}{tree_thinning}}
After passing to a further subsequence, we may suppose that every band
$B\notin \Lambda$ meets the support of $w_m$ for at most one value of $m$.
\end{lemma}

\begin{proof}
For every band $B\notin \Lambda$, let $\ell(B)$ be the last index $m$ for
which $P_B w_m\ne 0$.

Choose $m_1<m_2<\cdots$ inductively as follows. Having chosen $m_k$, let
$\mathcal F_k$ be the finite set of bands $B\notin \Lambda$ with
$P_B w_{m_k}\ne 0$, and choose
\[
m_{k+1}>\max\bigl(\{m_k\}\cup\{\ell(B):B\in \mathcal F_k\}\bigr).
\]
Then every band outside $\Lambda$ which appears at stage $m_k$ never appears
again. Passing to the subsequence $(w_{m_k})$ proves the claim.
\end{proof}

Write
\[
u_m:=P_{\Lambda} w_m,
\qquad
v_m:=w_m-u_m.
\]
By construction, the bands meeting $\supp v_m$ are disjoint from those
meeting $\supp v_{m'}$ whenever $m\ne m'$.

\begin{lemma}\label{lem:transient}
\leanformalization{\leanstatement{OrderClosures/WeaklyFatou/Moderated.lean}{94}{tree_transient}}
The finite sets $F_m:=\supp v_m$ are pairwise $\parent$-disjoint.
Consequently, the sequence $(T_n v_m)$ has no non-zero common lower bound.
\end{lemma}

\begin{proof}
If $s,t\in G_n\setminus\{\varnothing\}$ belong to the same sibling class,
then $\parent(s)=\parent(t)$. Thus, two supports can share a parent only if
they meet the same band $B_u$ for some $u\in H_n$. By
\Cref{lem:thinning}, this never happens for distinct $m$, so the sets $F_m$
are pairwise $\parent$-disjoint.

Let $z\in X_n^+$ be a lower bound of all $T_n v_m$. Put
$C:=2^n\sup_m \rho_n(v_m)$. Since every coefficient has weight at least
$2^{-n}$, we have
\[
T_n v_m \le C \sup_{t\in F_m} s_t.
\]
If $C=0$, then every $v_m$ is zero and there is nothing to prove. Hence,
\[
C^{-1}z \le \sup_{t\in F_m} s_t
\qquad (m\in\N).
\]
\Cref{lem:pi-disjoint} now gives $z=0$.
\end{proof}

The preceding lemmas allow us to finally bound the weak Fatou constant of $(X_n,p_n)$.

\begin{proposition}\label{prop:moderated}
\leanformalization{\leanstatement{OrderClosures/WeaklyFatou/Moderated.lean}{173}{component_moderated}}
For every increasing sequence $(x_m)$ in $X_n^+$ with $p_n(x_m)\le 1$ and
for every $\varepsilon>0$, there exists $y\in T_n[W_n^+]$ such that
\[
x_m\le y \quad (m\in\N),
\qquad
p_n(y)\le 2+\varepsilon.
\]
\end{proposition}

\begin{proof}
Start with the sequence constructed above and apply
\Cref{lem:sharp-subsequence,lem:trim,lem:thinning}. Whenever we pass to a
subsequence of $(w_m)$, we pass to the corresponding subsequence of
$(x_m)$ and relabel. Since the original sequence $(x_m)$ is increasing, any
upper bound for the relabelled subsequence is also an upper bound for the
original sequence.

The supports of $S_n u_m$ lie in the finite set consisting of the root
together with the parents of the finitely many bands in $\Lambda$. Hence
$(S_n u_m)$ lives in a finite-dimensional subspace of $W_n$. By boundedness
and \Cref{lem:upshift-basic}, after passing to another subsequence we may
assume that $S_n u_m$ converges in $W_n$ to some $u\in W_n^+$. Put
$y:=T_n u$.
Again by \Cref{lem:upshift-basic}, we have
\[
\rho_n(u)=\lim_m \rho_n(S_n u_m)
\le 2\sup_m \rho_n(u_m)
\le 2+\varepsilon,
\]
so
\[
p_n(y)\le \rho_n(u)\le 2+\varepsilon.
\]
Since $T_n$ is continuous on the finite-dimensional space containing all
$S_n u_m$, we have
\[
\|T_n S_n u_m-y\|_\infty \longrightarrow 0.
\]
Fix $j\in\N$ and $\eta>0$. Choose $m_0$ such that
\[
\|T_n S_n u_m-y\|_\infty\le \eta
\qquad (m\ge m_0).
\]
For $m\ge \max\{j,m_0\}$ we have
\[
x_j\le x_m\le T_n w_m = T_n u_m + T_n v_m \le T_n S_n u_m + T_n v_m
\le y+\eta\one + T_n v_m.
\]
Hence
\[
(x_j-y-\eta\one)^+ \le T_n v_m
\qquad (m\ge \max\{j,m_0\}).
\]
By \Cref{lem:transient}, this tail has no non-zero common lower bound.
Therefore
$(x_j-y-\eta\one)^+=0$, that is, $x_j\le y+\eta\one$.
Letting $\eta\downarrow 0$ gives $x_j\le y$. Since $j$ was arbitrary, $y$
dominates the whole sequence.
\end{proof}

\begin{corollary}\label{cor:weak-fatou}
\leanformalization{\leanstatement{OrderClosures/WeaklyFatou/Moderated.lean}{423}{component_weakFatou}}
The Banach lattice $(X_n,p_n)$ has weak Nakano constant $2$, hence weak
Fatou constant $2$.
\end{corollary}

\begin{proof}
The positive set $T_n[W_n^+]$ satisfies the hypothesis of
\Cref{prop:separable-reduction} by \Cref{prop:moderated}. Since $X_n$ is
separable, the conclusion follows.
\end{proof}



Let now
\[
L_n:=\{s_t: t\in \N^n\}
\]
be the set of characteristic functions of terminal nodes of the tree $G_n$.

\begin{proposition}\label{prop:component}
\leanformalization{\leanstatement{OrderClosures/WeaklyFatou/FinalSpace.lean}{21}{component_large_iterated_adherence}}
For each $n\in\N$:
\begin{enumerate}[label=\textup{(\roman*)}]
\item $(X_n,p_n)$ is a separable Banach lattice with weak Fatou constant $2$;
\item the constant function $\one$ satisfies
$2^n \one \in B_{X_n}^{\,o(n)}$;
\item $p_n(2^n\one)=2^n$.
\end{enumerate}
\end{proposition}

\begin{proof}
Part \textup{(i)} is \Cref{cor:Yn-basic,cor:weak-fatou}.

For \textup{(ii)}, \Cref{lem:exact-basis} gives
$p_n(2^n s_t)=1$ for every terminal $t$, so the solid hull
\[
A_n:=\operatorname{sol}(2^n L_n)
\]
is contained in $B_{X_n}$.
We claim, more generally, that for every node $t\in G_n$,
\[
2^n s_t \in A_n^{o(n-|t|)}.
\]
If $|t|=n$, this is tautological. If $|t|<n$, then by
\Cref{lem:basic-tree}\textup{(c)},
\[
2^n s_t=\sup_{m\in\N} 2^n s_{t^\frown m}.
\]
The children form an increasing sequence, so
$2^n s_t\in \bigl(A_n^{o(n-|t|-1)}\bigr)^o=A_n^{o(n-|t|)}$.
Taking $t=\varnothing$ yields
$2^n\one\in A_n^{o(n)}\subseteq B_{X_n}^{\,o(n)}$.

Part \textup{(iii)} follows from $p_n(\one)=1$.
\end{proof}


Finally, let us consider the space
\[
X:=c_0(X_n)_{n\in\N},
\]
equipped with the usual supremum  norm
\[
\|(x_n)\|:=\sup_n p_n(x_n).
\]

\begin{lemma}\label{lem:c0-weak-fatou}
\leanformalization{\leanstatement{OrderClosures/WeaklyFatou/FinalSpace.lean}{279}{finalSpace_weakFatou}}
$X$ is a Banach lattice with weak Fatou constant $2$.
\end{lemma}

\begin{proof}
Completeness is standard for $c_0$-sums. Let $0\le x_\alpha\uparrow x$ in
$X$. For each $n$, the $n$th coordinate satisfies
\[
p_n(x(n))\le 2\sup_\alpha p_n(x_\alpha(n))
\le 2\sup_\alpha \|x_\alpha\|.
\]
Taking the supremum over $n$ gives
\[
\|x\|\le 2\sup_\alpha \|x_\alpha\|.
\]
\end{proof}

For each $n$, let $z_n\in X$ be the vector whose $n$th coordinate is
$2^n\one\in X_n$ and whose other coordinates are $0$.

\begin{proposition}\label{prop:zn}
\leanformalization{\leanstatement{OrderClosures/WeaklyFatou/FinalSpace.lean}{507}{finalLargeVector_properties}}
For every $n\in\N$ one has
\[
z_n\in B_X^{\,o(n)}
\qquad\text{and}\qquad
\|z_n\|=2^n.
\]
\end{proposition}

\begin{proof}
By \Cref{prop:component},
$2^n\one\in B_{X_n}^{\,o(n)}$. Identifying $X_n$ with the $n$th coordinate
band of $X$, the same directed families witness $z_n\in B_X^{\,o(n)}$.
Also $\|z_n\|=p_n(2^n\one)=2^n$.
\end{proof}

\begin{proof}[Proof of \Cref{thm:fremlin-main}]
By \Cref{lem:c0-weak-fatou}, $X$ is weakly Fatou.

Assume, towards a contradiction, that $X$ admits an equivalent lattice norm
$\sigma$ with the Fatou property. Then there exist constants $c,C>0$ such
that
\[
c^{-1}\|x\|\le \sigma(x)\le C\|x\|
\qquad (x\in X).
\]
Hence
\[
B_X \subseteq C\,B_\sigma.
\]
By \Cref{lem:order-basic,lem:fatou-order},
\[
B_X^{\,o(n)}
\subseteq C\,B_\sigma^{\,o(n)}
= C\,B_\sigma
\subseteq Cc\,B_X
\qquad (n\ge 1).
\]
Thus, every element of every iterated order adherence of $B_X$ has
$\|\cdot\|$-norm at most $Cc$.
However, \Cref{prop:zn} gives $z_n\in B_X^{\,o(n)}$ and
\[
\|z_n\|=2^n \to \infty,
\]
a contradiction. Therefore, $X$ is not equivalent to any Fatou lattice norm.
\end{proof}

\section*{Acknowledgments}

We wish to thank A.~Wickstead for sharing the preprint \cite{Elliott} with us and for his guidance on how to credit M.~Elliott.

We are particularly grateful to J.~de Dios Pont for an earlier Lean formalization in March 2026 of the main result in Section \ref{sec:fremlin}. We later Lean verified all of the results in this paper on top of the Banach lattice Lean library, but his contribution was vital for convincing us that this would be possible, and, indeed, motivating the creation of the Banach lattice Lean library in the first place. The original Lean verification was done in the GPT 5.4/Opus 4.6 era; GPT 5.4 was also useful for expanding and understanding parts of the proof of \cite{Elliott}. The final Lean verification of the whole paper was done with GPT 5.6 Sol.

The first and last authors would like to express our sincere gratitude to the Taylor family for hosting us in their lakeside cabin near Edmonton in May 2023. Their exceptional care and hospitality provided the inspiration and initial discussions for the research presented in this work.

A.~Avil\'{e}s was supported by MICIU/AEI /10.13039/501100011033/ and ERDF-A way of making Europe (project PID2021-122126NB-C32). A.~Avil\'{e}s and P.~Tradacete were supported by Fundaci\'{o}n S\'{e}neca - ACyT Regi\'{o}n de Murcia. P.~Tradacete was partially supported by grants PID2024-162214NB-I00 and CEX2023-001347-S funded by MCIN/AEI/10.13039/501100011033.

\bibliographystyle{plain}
\bibliography{refs}

\end{document}